\documentclass[12pt]{article}%

\usepackage{mathrsfs}
\usepackage{latexsym}
\usepackage[all]{xy}
\usepackage{syntonly}
\usepackage{nomencl}
\usepackage[bbgreekl]{mathbbol}
\usepackage{bbm}
\usepackage{amssymb, amsthm, amsfonts, amsxtra, amsmath}

\usepackage{parskip}
\makeatletter 
\def\thm@space@setup{%
 \thm@preskip=\parskip \thm@postskip=0pt
}
\def\th@remark{%
  \thm@headfont{\itshape}%
  \normalfont 
  \thm@preskip\parskip \thm@postskip=0pt
}
\makeatother

\newtheorem{Theorem}{Theorem}[section]
\newtheorem{Def}[Theorem]{Definition}
\newtheorem{Lem}[Theorem]{Lemma}
\newtheorem{Lemma}[Theorem]{Lemma}
\newtheorem{Prop}[Theorem]{Proposition}

\newtheorem{Cor}[Theorem]{Corollary}

\newtheorem{Def-Prop}[Theorem]{Definition-Proposition}
\newtheorem{Not}[Theorem]{Notation}
\newtheorem{Exa}[Theorem]{Example}

\newtheorem{Rem}[Theorem]{Remark}

\newcommand{\Z}{\mathbbm{Z}}
\newcommand{\N}{\mathbbm{N}}
\newcommand{\C}{\mathbbm{C}}
\newcommand{\R}{\mathbbm{R}}
\newcommand{\id}{\mathrm{id}}
\newcommand{\End}{\mathrm{End}}
\newcommand{\Mor}{\mathrm{Mor}}
\newcommand{\Fus}{\mathrm{Fus}}

\newcommand{\Gam}{\mathbf{\Gamma}}
\newcommand{\Lam}{\mathbf{\Lambda}}
\newcommand{\Corep}{\mathrm{Corep}}

\newcommand{\Unit}{\mathbbm{1}}

\title{Torsion-freeness for fusion rings and tensor C$^*$-categories}%
\author{Yuki Arano\thanks{Graduate School of Mathematics, University of Tokyo, 3-8-1, Komaba, Meguro-ku, Tokyo, Japan, email: {\tt arano@ms.u-tokyo.ac.jp}.  Supported by the Research Fellow of the Japan Society for the Promotion of Science and the Program for Leading Graduate Schools, MEXT, Japan.}\;\, and Kenny De Commer\thanks{Department of Mathematics, Vrije Universiteit Brussel, VUB, B-1050 Brussels, Belgium, email: {\tt kenny.de.commer@vub.ac.be}. This work was partially supported by FWO grant G.0251.15N.}}
\date{}%

\begin{document}

\maketitle

\begin{abstract}\noindent Torsion-freeness for discrete quantum groups was introduced by R. Meyer in order to formulate a version of the Baum-Connes conjecture for discrete quantum groups. In this note, we introduce torsion-freeness for abstract fusion rings. We show that a discrete quantum group is torsion-free if its associated fusion ring is torsion-free. In the latter case, we say that the discrete quantum group is \emph{strongly} torsion-free. As applications, we show that the discrete quantum group duals of the free unitary quantum groups are strongly torsion-free, and that torsion-freeness of discrete quantum groups is preserved under Cartesian and free products. We also discuss torsion-freeness in the more general setting of abstract rigid tensor C$^*$-categories.\end{abstract}

\section*{Introduction}

R. Meyer introduced in \cite{Mey08}  the notion of torsion-freeness for discrete quantum groups in order to state a plausible Baum-Connes conjecture for them. In his definition, a discrete quantum group is torsion-free if and only if any action of its dual compact quantum group  on a finite-dimensional C$^*$-algebra is equivariantly Morita equivalent to a trivial action. 

For the duals of the $q$-deformations of semi-simple compact Lie groups, torsion-freeness was shown in \cite{Gof12}, based on the fundamental computation in \cite{Voi11} for the dual of quantum $SU(2)$ \cite{Wor87}. By means of monoidal equivalence \cite{BDV05}, this result was then extended to certain free orthogonal quantum groups introduced in \cite{VDW96}. However, up till now it was not known if also the free unitary quantum groups are torsion-free \cite{VV11}. 

The main result of this note is that the duals of free unitary quantum groups are indeed torsion-free. Our result will in fact show that they are torsion-free in a strong way. For this, recall that to any discrete quantum group can be associated a \emph{fusion ring}, which remembers how tensor products of irreducible representations of its dual compact quantum group decompose into irreducibles. We then propose a definition of torsion-freeness for fusion rings in terms of \emph{cofinite fusion modules}, and show that any discrete quantum group with a torsion-free fusion ring is automatically torsion-free. We hence call a discrete quantum group \emph{strongly} torsion-free if its associated fusion ring is torsion-free. By a combinatorial argument, it is shown that the duals of free unitary quantum groups are strongly torsion-free.

The notion of (strong) torsion-freeness can in fact be captured purely within the language of rigid tensor C$^*$-categories. Although the case of discrete quantum groups can be subsumed in this more general setting, we treat it separately as to be more accessible to people unfamiliar with tensor C$^*$-categories. 

The structure of this paper is as follows. In the \emph{first section}, we define torsion-freeness for fusion rings, and show that the fusion rings coming from free unitary quantum groups are torsion-free. We also show that torsion-freeness is preserved by free products, and discuss the case of tensor products. In the \emph{second section}, we show torsion-freeness of discrete quantum groups with torsion-free fusion ring, and deduce that the duals of free unitary quantum groups are torsion-free. In the \emph{third section}, we discuss torsion-freeness for general rigid tensor C$^*$-categories, and show that Cartesian and free products of torsion-free quantum groups are torsion-free. 

\emph{Acknowledgments}: We thank C. Voigt for discussions and for bringing to our attention the problem of torsion-freeness for free unitary quantum groups. We also thank M. Bischoff for comments on a first draft of this article. The first author would like to express his gratitude to the University of Glasgow and KU Leuven for their invitation and hospitality.

\section{Torsion-free fusion rings}

Let $I$ be a pointed set with distinguished element $\Unit$, equipped with an involution $\sigma: \alpha \mapsto \overline{\alpha}$ fixing $\Unit$. Consider the free $\Z$-module $\Z_I$ with basis $I$, where we write addition by $\oplus$. Endow $\Z_I$ with the $\Z$-linear functional $\tau$ such that $\tau(\alpha) = \delta_{\alpha,\Unit }$ and extend the involution on $I$ to a $\Z$-linear involution $\sigma: x\mapsto \overline{x}$.

\begin{Def}\label{DefBasRing}
We call \emph{$I$-based ring} \cite{Lus87,EK95,Ost03b} any ring-structure $\otimes$ on the abelian group $\Z_I$, with unit $\Unit$ and such that, for all $\alpha,\beta,\gamma\in I$,
\begin{itemize}
\item $\overline{\alpha\otimes \beta} = \overline{\beta}\otimes \overline{\alpha}$, 
\item $\tau(\overline{\alpha}\otimes \beta)  = \delta_{\alpha,\beta}$,
\item $\tau(\alpha\otimes \beta \otimes\gamma)\in \N$. 
\end{itemize}
We call \emph{fusion ring} (over $I$) an $I$-based ring $(\Z_I,\otimes)$ which admits a \emph{dimension function}, that is, a unital ring homorphism $d:(\Z_I,\otimes) \rightarrow \R$ such that, for all $\alpha\in I$,
\begin{itemize}
\item $d(\alpha)>0$,
\item $d(\alpha) = d(\overline{\alpha})$. 
\end{itemize}
\end{Def}

For $(\Z_I,\otimes)$ an $I$-based ring and $x,y,z\in \Z_I$, write \begin{equation}N_{yz}^{x} = \tau(y\otimes z\otimes \overline{x}).\end{equation} Then it follows that, for $x,y\in \Z_I$ fixed, $N_{xy}^{\alpha} =0$ for almost all $\alpha\in I$, and \begin{equation}x\otimes y = \oplus_{\alpha} N_{xy}^{\alpha}  \alpha.\end{equation} 

Note that by the second conditions, $d(\alpha)\geq 1$ for $\alpha \in I$ in a fusion ring with dimension fuction $d$. 

When $x,y\in \Z_I$, we write $x\leq y$ if $y-x$ is an $\N$-linear combination of elements in $I$. 

\begin{Exa} Let $\Gamma$ be a discrete group. Consider $\Gamma$ with $\Unit$ the unit of $\Gamma$ and $\overline{g}=g^{-1}$. Then $\Z_{\Gamma}$ is a fusion ring with $g\otimes h = gh$ and dimension function $d(g)= 1$.  
\end{Exa} 

\begin{Exa} Let $\phi$ be the golden ratio. Then the ring $\Z\lbrack \phi\rbrack$ becomes a fusion ring by giving it the $\Z$-basis $\{1,\phi\}$ and trivial involution. Its dimension function is given by the inclusion $d:\Z\lbrack \phi\rbrack \hookrightarrow \R$.  
\end{Exa}

\begin{Exa}\label{ExaA1} Consider $\N=\{\mathbf{0},\mathbf{1},\ldots\}$ with distinguished element $\mathbf{0}$ and trivial involution. We can endow $\Z_\N$ with the fusion ring structure \[\mathbf{n} \otimes \mathbf{m} = \underset{k\in \{|n-m|,|n-m|+2,\ldots, n+m\}}{\oplus} \mathbf{k}\] and dimension function $d(\mathbf{n}) = n+1$. We will write this fusion ring as $A(1)$. 
\end{Exa}

\begin{Exa}\label{ExaA2} Let $F$ consist of words in the variables $\pi_+,\pi_-$ with distinguished element the empty word and the unique involution which inverts word order and sends $\pi_+$ to $\pi_-$. We can endow $\Z_F$ with the fusion ring structure \[w\otimes z = \underset{xu = w, \overline{u}y = z} {\sum_{\exists u\in F}}xy\] and the dimension function uniquely determined by $d(\pi_+) = 2$. We will write this fusion ring as $A(2)$. 
\end{Exa}

\begin{Exa} Let $(\Z_{I_1},\otimes)$ and $(\Z_{I_2},\otimes)$ be fusion rings with respective dimension functions $d_i$. Write $\odot$ for the tensor product over $\Z$. Then $\Z_{I_1}\odot \Z_{I_2}$ is a fusion ring with basis $\alpha\odot \beta$ for $\alpha \in I_1,\beta\in I_2$, with unit $\Unit  = \Unit _1\odot \Unit _2$ and involution $\overline{x\odot y} = \overline{x}\odot \overline{y}$. It has a dimension function given by $d(\alpha\odot \beta) = d_1(\alpha)d_2(\beta)$. We call it the \emph{tensor product} of $(\Z_{I_1},\otimes)$ and $(\Z_{I_2},\otimes)$. 
\end{Exa}

\begin{Exa} Let $(\Z_{I_i},\otimes)$ and $(\Z_{I_2},\otimes)$ be fusion rings with respective dimension functions $d_i$. Let $I_1*I_2$ consist of alternating words, possibly empty, in $I_1\setminus\{\Unit _1\}$ and $I_2\setminus\{\Unit _2\}$. We view $I_1*I_2$ as a pointed set with distinguished element $\Unit $ the empty word, and endow it with the involution inverting order and acting as the involution of $I_i$ on each letter. Then we can define a unique $I_1*I_2$-based fusion ring with unit $\Unit $ such that \[w\alpha\otimes \beta z = w\alpha\beta z, \quad \alpha\in I_i\setminus\{\Unit _i\},\beta\notin I_i\sqcup \{\Unit _2\},\]
\[w\alpha\otimes \beta z=  N_{\alpha,\beta}^{\Unit _i}  (w\otimes z) \oplus \left(\oplus_{\gamma\in I_i\setminus\{\Unit _i\}} N_{\alpha,\beta}^{\gamma} w\gamma z\right),\quad \alpha,\beta\in I_i\setminus\{\Unit _i\},\]and dimension function uniquely determined by $d_{\mid I_i\setminus\{\Unit _i\}} = d_i$.  We call it the \emph{free product} of $(\Z_{I_1},\otimes)$ and $(\Z_{I_2},\otimes)$
\end{Exa}

Evidently, the above two constructions can be performed with respect to an arbitrary number of components as well, viewed for example as inductive limits in the case of infinite index sets.

\begin{Exa} Let $(\Z_I,\otimes)$ be a fusion ring, and assume $I' \subseteq I$ is a pointed subset which is invariant under the involution and for which $N_{\alpha,\beta}^{\gamma} =0$ for all $\alpha,\beta\in I'$ and $\gamma \in I\setminus I'$. Then we obtain by restriction of $\otimes$ a fusion ring $(\Z_{I'},\otimes)$, which is called a \emph{fusion subring} of $(\Z_I,\otimes)$. 
\end{Exa}

For example, if $(\Z_I,\otimes)$ is a fusion ring and $\alpha \in I$, one can consider the fusion subring \emph{generated by} $\alpha\in I$, which is the smallest fusion subring of $(\Z_I,\otimes)$ containing $\alpha$. 

\begin{Rem} Clearly, an $I$-based ring is completely determined by $(I,\Unit ,\sigma)$ and a collection of numbers $N_{\beta\gamma}^{\alpha} \in \N$ satisfying the following.
\begin{enumerate}
\item $N_{\beta \gamma}^{\alpha} = N_{\gamma \overline{\alpha}}^{\overline{\beta}} = N_{\overline{\alpha} \beta}^{\overline{\gamma}}= N_{\overline{\gamma}\overline{\beta}}^{\overline{\alpha}}$ for all $\alpha,\beta,\gamma$. 
\item $N_{\Unit \beta}^{\alpha} = \delta_{\alpha,\beta}$ for all $\alpha,\beta$. 
\item For all $\beta,\gamma$ fixed, $N_{\beta\gamma}^{\alpha}=0$ for all but a finite number of $\alpha$.
\item $\sum_{\epsilon} N_{\beta\gamma}^{\epsilon} N_{\alpha\epsilon}^{\delta} = \sum_{\epsilon} N_{\alpha\beta}^{\epsilon}N_{\epsilon\gamma}^{\delta}$ for all $ \alpha,\beta,\gamma,\delta$.
\end{enumerate}
\end{Rem}\vspace{0.2cm}

\begin{Def} Let $J$ be a set. We call \emph{cofinite based module (over $J$)} for a based ring $(\Z_I,\otimes)$ a (unital) left $(\Z_I,\otimes)$-module structure $\otimes$ on $\Z_J$, together with a $\Z_I$-valued bilinear form $\langle -,-\rangle$ satisfying, for all $\alpha\in I$ and $b,c\in J$,  
\begin{itemize}
\item  $\langle \alpha\otimes b,c\rangle = \alpha\otimes \langle b,c\rangle$, 
\item $\langle b,c\rangle = \overline{\langle c,b\rangle}$,
\item $\tau(\langle b,c\rangle)  = \delta_{b,c}$,
\item $\tau(\alpha\otimes \langle b,c\rangle) \in \N$.  
\end{itemize}
\end{Def}

If $(\Z_I,\otimes)$ is a fusion ring, we will also call $(\Z_J,\otimes)$ a cofinite \emph{fusion module}. If then $d$ is a dimension function on $(\Z_I,\otimes)$, we say a $\Z$-linear function $d: \Z_J\rightarrow \R$ is a \emph{$d$-compatible dimension function} if
\begin{itemize}
\item $d(a)>0$ for all $a\in J$,
\item $d(\alpha\otimes a) = d(\alpha)d(a)$ for all $\alpha \in I,a\in J$.
\end{itemize}
Clearly there exists at least one dimension function on $(\Z_J,\otimes)$ by defining \[d(b) := d(\langle b,b_0\rangle)\] for some fixed element $b_0\in J$.

For $(\Z_J,\otimes)$ a cofinite based $(\Z_I,\otimes)$-module, we write \begin{equation} N_{x y}^{z} = \tau(x\otimes \langle y,z\rangle),\quad x\in \Z_I,y,z\in \Z_J.\end{equation} Then for $x\in \Z_I,y\in \Z_J, c\in J$,  \begin{equation}\label{EqMod}x \otimes y = \oplus_{c} N_{x y}^c c.\end{equation}

Again, we write $x\leq y$ if $y-x$ is an $\N$-linear combination of elements in $\Z_J$.

\begin{Exa} Let $A= (\Z_I,\otimes)$ be a unital $I$-based ring. Then $A$ is a cofinite $I$-based module over itself by left $\otimes$-multiplication and inner product \[\langle  \alpha,\beta\rangle = \alpha \otimes \overline{\beta}.\] We will call this the \emph{standard} $I$-based module. 
\end{Exa}

\begin{Rem}\label{RemFinMod} Clearly, a cofinite $J$-based module is completely determined by numbers $N_{\alpha b}^{c}\in \N$ satisfying the following.
\begin{enumerate}
\item For $b,c\in J$ fixed, $N_{\alpha b}^c=0$ for all but a finite number of $\alpha\in I$.
\item For $\alpha \in I,b\in J$ fixed, $N_{\alpha b}^c=0$ for all but a finite number of $c\in J$.
\item $N_{\Unit b}^c = \delta_{b,c}$ for all $b,c\in J$.
\item $N_{\alpha b}^c = N_{\overline{\alpha} c}^b$ for all $b,c\in J$ and $\alpha \in I$.
\item $\sum_{e} N_{\beta c}^{e} N_{\alpha e}^{d} = \sum_{\epsilon} N_{\alpha\beta}^{\epsilon}N_{\epsilon c}^{d}$ for all $\alpha,\beta \in I,c,d\in J$.
\end{enumerate}
The associated module structure and inner product is determined by \begin{equation} \alpha \otimes b = \sum_c N_{\alpha b}^c c,\qquad \langle b,c\rangle = \oplus_{\alpha} N_{\alpha b}^c \overline{\alpha}.\end{equation}
\end{Rem}
 
\begin{Rem} If all conditions in Remark \ref{RemFinMod} are satisfied except possibly the first, then $\Z_J$ is still a $(\Z_I,\otimes)$-module by \eqref{EqMod}. In this case, we call $(\Z_J,\otimes) $ a (general) based $(\Z_I,\otimes)$-module.
\end{Rem} 

\begin{Lem}\label{LemNoZero} Let $(\Z_J,\otimes)$ be a based $(\Z_I,\otimes)$-module. Then $\forall \alpha\in I,b\in J$, we have $\alpha\otimes b\neq 0$.
\end{Lem} 
\begin{proof} If $\alpha\otimes b =0$, then also $\overline{\alpha}\otimes \alpha \otimes b=0$. But since $\Unit \leq \overline{\alpha}\otimes \alpha$, it follows that  $1= N_{\Unit ,b}^b \leq N_{\overline{\alpha}\otimes \alpha,b}^b = 0$, a contradiction.
\end{proof}

\begin{Def} We call a $J$-based $(\Z_I,\otimes)$-module \emph{connected} if for any $b,c\in J$, there exists $\alpha \in I$ with $N_{\alpha b}^c\neq 0$.
\end{Def} 

By the fourth identity in Remark \ref{RemFinMod}, it is enough to require the above condition for a single fixed $b$. This immediately implies that any (cofinite) based $(\Z_I,\otimes)$-module is a (possibly infinite) direct sum of connected (cofinite) based modules.

\begin{Lemma}\label{LemCof} Let $(\Z_J,\otimes)$ be a connected based $(\Z_I,\otimes)$-module. Assume that $b,c\in J$ are such that $N_{\alpha b}^c=0$ for all but a finite number of $\alpha$. Then $(\Z_J,\otimes)$ is cofinite. 
\end{Lemma} 

\begin{proof} For $\beta\in I$ fixed, $N_{\alpha, \beta\otimes b}^{c} =N_{\alpha\otimes \beta,b}^c =0$ for all but a finite number of $\alpha$. Take now $d\in J$ arbitrary. Then by connectedness, we can find $\beta\in I$ with $d\leq \beta\otimes b$. As $N_{\alpha d}^{c} \leq N_{\alpha, \beta\otimes b}^{c}$, it follows that $N_{\alpha d}^c=0$ for all but a finite number of $\alpha$. Since $N_{\alpha d}^c = N_{\overline{\alpha} c}^d$, we can replace also $c$ by an arbitrary element in $J$ to conclude that  $(\Z_J,\otimes)$ is cofinite. 
\end{proof}

To elements of a fusion ring acting on a based module, one can associate combinatorial data in the following way. 

\begin{Not} If $(\Z_J,\otimes)$ is a based $(\Z_I,\otimes)$-module, we write $M_J(\alpha)$ for the matrix $M_J(\alpha)_{b,c} = N_{\alpha b}^c$. We write $\Gamma_J(\alpha)$ for the associated (oriented) graph having $M_J(\alpha)$ as its adjacency matrix, so that there are $M_J(\alpha)_{b,c}$ arrows from $b$ to $c$. 
\end{Not} 

Note that on each row/column of $M_J(\alpha)$, there are only finitely many non-zero entries, which are then positive integers. It follows that we can apply $M(\alpha)$ to any vector $v\in \C^J$. Note also that $M(\alpha\otimes \beta) = M(\alpha)M(\beta)$ and $M(\overline{\alpha}) = M(\alpha)^*$. 

If $(\Z_J,\otimes)$ is a fusion module over the fusion ring $(\Z_I,\otimes)$, with a pair of compatible dimension functions $d$, we can view $d$ as a vector $D\in \C^J$ with $D_b = d(b)$. It follows that for any $\alpha \in I$, \begin{equation}\label{EqDimFor} (M_J(\alpha)D)_{c} = \sum_{b} M_J(\alpha)_{c,b} d(b) = d(\oplus_{b} N_{\mathbf{\alpha} c}^{b} b) = d(\mathscr{\alpha}\otimes c) = d(\mathscr{\alpha})D_c.\end{equation} As $D$ is also an eigenvector at this eigenvalue for $M_J(\overline{\alpha})$, we obtain by the Schur test that $M_J(\alpha)$ is a bounded matrix on $l^2(J)$ with $\|M_J(\alpha)\| \leq  d(\alpha)$.

The following is our main definition.

\begin{Def} We call a fusion ring \emph{torsion-free} if any non-zero connected cofinite based module is isomorphic to the standard based module.
\end{Def}

In the above definition, an isomorphism of based modules is assumed to take basis elements to basis elements. 

\begin{Prop}\label{PropGroup} Let $\Gamma$ be a discrete group. Then $\Gamma$ is torsion-free if and only if $(\Z_{\Gamma},\otimes)$ is torsion-free. 
\end{Prop} 

\begin{proof} This follows immediately from the fact that a non-zero cofinite $J$-based $(\Z_{\Gamma},\otimes)$-module is determined by an action $\Gamma \curvearrowright J$ with finite stabilizers.
\end{proof}

\begin{Prop}\label{PropSubFus}  Assume $(\Z_I,\otimes)$ is a torsion-free fusion ring with dimension function $d$. Then $\Gamma = \{g\in I\mid d(g) = 1\}$ is torsion-free. 
\end{Prop}
\begin{proof} Note that $(\Gamma,\otimes)$ is a discrete group with a natural right action on $\Z_I$ by $\otimes$.  If $\Gamma$ is not torsion-free, let $H$ be a non-trivial finite subgroup of $\Gamma$. Then $(\Z_{I/H},\otimes)$ is a non-zero connected cofinite fusion module for $(\Z_I,\otimes)$ in a natural way, hence isomorphic to $(\Z_{I},\otimes)$ as a based $(\Z_I,\otimes)$-module. Since this isomorphism must preserve basis elements of minimal dimension, it sends $\Gamma/H$ to $\Gamma$ in a $\Gamma$-equivariant way, which gives a contradiction.
\end{proof}

\begin{Prop} If $(\Z_I,\otimes)$ is a fusion ring with finite $I$ and integer-valued dimension fuction, then $(\Z_I,\otimes)$ is torsion-free if and only if $I$ is a singleton.
\end{Prop} 

\begin{proof} Assume $(\Z_I,\otimes)$ is a fusion ring with integer-valued dimension function $d$. Then we can endow $\Z = \Z_{\{\bullet\}}$ with the non-zero cofinite $\{\bullet\}$-based $(\Z_I,\otimes)$-module structure $\alpha \otimes \bullet = d(\alpha)\bullet$. Hence if $(\Z_I,\otimes)$ is torsion-free, we must have $I \cong \{\bullet\}$.
\end{proof} 

\begin{Prop}\label{PropTorFree} The fusion ring $(\Z[\phi],\otimes)$ is torsion-free.
\end{Prop}
\begin{proof} Any connected fusion module for $(\Z[\phi],\otimes)$  is completely determined by $M=M_J(\phi)$, which is a symmetric, positive integer-valued matrix with norm $\phi$, since $M^2 = M+1$. As $M$ determines a connected graph, it follows by the classification of graphs with small norms (see e.g. \cite{GHJ89}) that necessarily (up to permutation)  $M$ is the tadpole $\begin{pmatrix} 0 & 1 \\1 & 1\end{pmatrix}$, that is, $M$ determines the standard module.   
\end{proof}

The above result shows that finite fusion rings can be non-trivial yet torsion-free. 

We now aim to show that the fusion rings $A(1)$ and $A(2)$ in Examples \ref{ExaA1} and \ref{ExaA2} are torsion-free. 

\begin{Prop}\label{TheoA1} The fusion ring $A(1)$ is torsion-free.
\end{Prop} 

\begin{proof} Let $(\Z_J,\otimes)$ be a non-zero cofinite fusion $A(1)$-module with chosen compatible dimension function $d(b) = d(\langle b,b_0\rangle)$ with respect to some fixed $b_0\in J$. Write $M=M_J(\mathbf{1})$. Let $\Lambda$ be the unoriented graph associated with $M$. Then $\|\Lambda\|=\|M\|=2$, and $\Lambda$ is connected by connectedness of $(\Z_J,\otimes)$ and the fact that $\mathbf{1}$ is generating for $A(1)$. Hence $\Lambda$ is an extended Dynkin diagram, possibly with loops. For convenience, we refer to the list of these graphs in \cite[Appendix]{Tom08}. 

By immediate inspection, all extended Dynkin diagrams except for $A_{\infty}$ have a Frobenius-Perron eigenvector at eigenvalue 2 which is uniformly bounded. Hence, suppose that $\Lambda$ is not $A_{\infty}$. By unicity of the Frobenius-Perron eigenvector up to a scalar multiple, it follows that $d$ must be uniformly bounded on $J$. However, for $b\in J$ we have \[d(b) = d(\langle b,b_0\rangle) = \sum_{n\geq 0} N_{\mathbf{n},b}^{b_0}d(\mathbf{n}) = \sum_{n\geq 0} (n+1)N_{\mathbf{n},b}^{b_0}.\] Hence, we infer that there exists $m\geq 0$ such that $N_{\mathbf{n},b}^{b_0}=0$ for all $b\in J$ and all $n\geq m$. It follows that $M = \langle \Z_J,b_0\rangle$ is a non-trivial finite rank submodule of $A(1)$. Clearly, this is impossible.

It follows that $\Lambda$ is an $A_{\infty}$-graph, and it is then immediate to conclude that $(\Z_J,\otimes)$ is isomorphic to the standard based $A(1)$-module, since the action of $\mathbf{1}$ determines the complete module. 
 \end{proof}

\begin{Theorem}\label{TheoA2} The fusion ring $A(2)$ is torsion-free. 
\end{Theorem} 

\begin{proof}  Let $(\Z_J,\otimes,d)$ be a non-zero cofinite fusion $A(2)$-module with chosen compatible dimension function $d$, which we may assume integer-valued. Write $M(\pm)=M_J(\pi_{\pm})$, $\Gamma(\pm) = \Gamma_J(\pi_{\pm})$. It is enough to prove that $\Gamma(+) = \Gamma_{F}(\pi_+)$, the graph associated to $\pi_+$ for the standard module.

Since $\pi_+^{\otimes n}  = \pi_+^n$, it follows that $N_{\pi_+,b}^b\neq 0$ for some $b\in J$ implies $N_{\pi_+^n,b}^{b}\neq 0$ for all $n \in \N$. By cofiniteness of $(\Z_J,\otimes)$, it follows that $N_{\pi_+,b}^b=0$ for all $b\in J$, and $\Gamma(+)$ has no loops.

Write $\widetilde{J} = J\times \{-,+\}$, whose elements we will write as $b_{\pm}$.
Let $\widetilde{\Lambda}$ be the unoriented graph with vertex set $\widetilde{J}$  and $M(\pm)_{b,c}$ edges between $b_{\pm}$ and $c_{\mp}$, and no other edges.  Endow $\Z\lbrack \widetilde{J}\rbrack$ with the unique based $A(1)$-module structure such that $N_{\mathbf{1},c_{\mu}}^{d_{\nu}} = \widetilde{M}_{c_{\mu},d_{\nu}}$, with $\widetilde{M}$ the (self-adjoint) adjacency matrix of $\widetilde{\Lambda}$.

As $(M(\pi_{\mu_{n}})\ldots M(\pi_{\mu}) M(\pi_{-\mu})M(\pi_{\mu}))_{b,c}$ counts the number of paths of length $n$ from $c$ to $b$ in $\Gamma(\pi)$ which alternate orientation and start in the direction $\pi_{\mu}$, it follows that this number equals $(\widetilde{M}^n)_{b_{\mu_n},c_{\mu}}$. But by induction, one easily verifies by use of the fusion rules that there exist $P_k \in \Z$ such that \[\pi_{\mu_{n}}\ldots\pi_{\mu} \pi_{-\mu}\pi_{\mu} = \sum_{k=0}^n P_k \pi_{\mu_{k}}\otimes\ldots \otimes \pi_{\mu} \otimes \pi_{-\mu}\otimes \pi_{\mu},\quad \mathbf{n} = \sum_{k=0}^n P_k \mathbf{1}^{\otimes k}.\]
We conclude that \[N_{\mathbf{n},b_{\mu_n}}^{c_{\mu}} = N_{\pi_{\mu_n}\ldots \pi_{\mu}\pi_{-\mu}\pi_{\mu},b}^c,\qquad b,c\in J.\] In particular, $N_{\mathbf{n},b_{\mu_n}}^{c_{\mu}} =0$ for $n$ large. We deduce that $\Z\lbrack \widetilde{J}\rbrack$ is a cofinite based $A(1)$-module, and so $\widetilde{\Lambda}$ is a disjoint union of $A_{\infty}$-graphs. 

It follows from the above that $\Gamma(+)$ can not contain any double arrows $\rightrightarrows$ or $\leftrightarrows$ between two vertices, and that for any $b\in J$, there are at most 2 edges incoming and at most 2 edges outgoing from any vertex $b$. In fact, the situation 
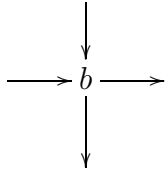
\[\xymatrix{ & \ar[d]& \\ \ar[r] & b \ar[r] \ar[d]&\\ &&}\] can not appear either. Indeed, this would imply \[b \oplus (\pi \rho \otimes b) = \pi \otimes \rho \otimes b = 2b \oplus \ldots,\] and hence $N^b_{\pi\rho,b} = 1$. Also $N^b_{\rho\pi,b} = 1$. But  $\pi\rho^2\pi^2\rho^2\ldots = \pi\rho\otimes \rho\pi\otimes \pi\rho\otimes \ldots$, so $N^b_{\alpha,b}$ is nonzero for all alternating products $\alpha$ of $\pi\rho$ and $\rho\pi$, contradicting the cofiniteness condition. 

Consider now $\Lambda$, the unoriented graph underlying $\Gamma(+)$. (There is no ambiguity in this definition as $\Gamma(+)$ does not have loops or multiple edges.) We will show that $\Lambda$ is in fact a tree.

Let $b_0$ be a vertex of minimal dimension. Then in $\Gamma(+)$ there can not be two edges leaving or two edges arriving in $b_0$, as one of their endpoints would have dimension strictly smaller than $d(b_0)$. As there has to be at least one edge arriving and one edge leaving, by Lemma \ref{LemNoZero}, the situation around $b_0$ is as follows: $b \leftarrow b_0 \rightarrow c$, with $d(b) = d(c) = 2d(b_0)$. 

Let $b_0 \overset{e_1}{-} b_1\overset{e_2}{-}\ldots \overset{e_n}{-} b_n$ be a path in $\Lambda$. We claim that $d(b_i) < d(b_{i+1})$. We have already shown that $d(b_0) < d(b_1)$. Suppose then that $d(b_{i-1})< d(b_i)$ for $i\geq 1$. By this dimensional assumption, there must be another edge with vertex $b_i$ of the same orientation as $e_{i}$. As there also has to be an edge to $b_i$ of the opposite orientation, we must be in one of the following situations: 
\begin{equation}\label{EqDiag}\begin{array}{lll} \xymatrix{ & \ar[d]& \\ b_{i-1} & \ar[l] b_i \ar[r] & b_{i+1}}  && \xymatrix{ & & \\ b_{i-1} & \ar[l] \ar[u] b_i & \ar[l] b_{i+1}}\\
 \\ \xymatrix{ & & \\ b_{i-1} \ar[r]& b_i \ar[u]& \ar[l] b_{i+1}} & &\xymatrix{ & \ar[d]& \\ b_{i-1} \ar[r]& b_i \ar[r]&  b_{i+1}}\end{array}\end{equation} In each of these situations, $d(b_{i+1})>d(b_i)$. It now follows immediately that $\Lambda$ must be a tree. 

From the above, we can already conclude that $J \cong F$, and that $\Lambda$ coincides with the unoriented graph underlying $\Gamma_{F}(\pi_+)$. Let us show that in fact $\Gamma(+) = \Gamma_{F}(\pi_+)$. Indeed, let again $b_0$ be a vertex of minimal dimension, viewed as the root of the tree $\Gamma(+)$. Then from level 0 to level 1, we have the situation 
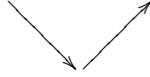
\[\xymatrix{ \ar[dr] && \\ & \ar[ur]&}\] By induction, it is clear from the only possible choices \eqref{EqDiag} that this pattern repeats itself from level $n$ to level $n+1$ at each vertex. This concludes the proof. 
\end{proof}

We now prove that free products of torsion-free fusion rings are torsion-free. 

\begin{Theorem}\label{TheoFreeProd} Let $A_s= (\Z_{I_s},\otimes)$ be torsion-free fusion rings for $s \in S$. Then their free product $A = *_{s \in S} A_s$ is again torsion-free.	
\end{Theorem}

\begin{proof}
Let $B=(\Z_J,\otimes)$ be a non-zero connected cofinite based fusion module for $A$ with $d=\underset{s}{*}d_s$-compatible dimension function $d(b) = d(\langle b,b_0\rangle)$ for some $b_0\in J$. For each $b \in J$, let $B_s^b$ with basis $J_s^b$ be the connected cofinite based $A_s$-submodule of $B$ spanned by $b$. By assumption, $B^b_s \cong A_s$ as $A_s$-based module for any $b \in J$.

\textbf{Claim}: There exists $e\in J$ such that $\beta \otimes \alpha\otimes e \in J$ for each $\beta\in \sqcup_s I_s$ and each $\alpha\in I$ with $d(\alpha)=1$. 

Indeed, assume that $b\in J$ and $\beta\otimes b \notin J$ for some $\beta \in I_s$. Choose $c\in B^b_s$ such that $\gamma\otimes c$ is irreducible for each $\beta\in I_s$. With $\gamma$ such that $\gamma\otimes c = b$, we must necessarily have that $d_s(\gamma)>1$. In fact, since $d_s(\gamma)$ is bounded below by the norm of a matrix of positive integers, we must then have $d_s(\gamma) \geq \sqrt{2}$, and so $d(b) \geq \sqrt{2}d(c)$. Replacing $b$ by $\alpha\otimes c$ for $\alpha\in I$ with $d(\alpha)=1$, we can iterate this argument, which must stop at some point as $d(a)\geq 1$ for all $a\in J$. Hence there must exist $e$ as in the claim.

Fix now $e\in J$ satisfying the above property. Then by its definition, it satisfies the 0th step in the following induction argument.

\textbf{Claim}: For any reduced word $\alpha = \alpha_0 \alpha_1 \ldots \alpha_n$ where $\alpha_i \in I_{s(i)}$, we have $\alpha \otimes e \in J$. Moreover for any $s \neq s(0)$, we have an isomorphism $\varphi: B^{\alpha \otimes e}_s \to A_s$ of based $A_s$-modules such that $\varphi(\alpha \otimes e) = \Unit _s$. 
	
Indeed, take any $\alpha$ as above and assume the claim holds for $\alpha^k = \alpha_k \alpha_{k+1} \ldots \alpha_n$ for each $k \geq 1$. Thanks to the isomorphism $\varphi': B^{\alpha^1 \otimes e}_{s(0)} \to A_{s(0)}$ as above, we get $\varphi'(\alpha \otimes e) = \alpha_0\in J_{s(0)}$, and hence $\alpha\otimes e \in J_{s(0)}^{\alpha^{1}\otimes e}\subseteq J$.

To show that $\varphi: B^{\alpha \otimes e}_s \to A_s$ exists, note that there necessarily exists an isomorphism $\varphi_0: B^{\alpha \otimes e}_s \to A_s$ of based $A_s$-modules, by the first part of the proof. It hence suffices to show \[d_s(\varphi_0(\alpha \otimes e)) = 1.\] Assume not.	Then there exists $\Unit _s \neq \beta \in I_s$ such that $N^{\alpha \otimes e}_{\beta,\alpha \otimes e}$ is nonzero. Moreover, $d(\alpha) \neq 1$ by our assumption and the defining property of $e$. Pick a minimal $k$ such that $d_{s(k)}(\alpha_k) \neq 1$. Again we have $\Unit _{s(k)}\neq \gamma_0 \in A_{s(k)}$ such that $N^{\alpha^k \otimes e}_{\gamma_0,\alpha^k \otimes e}$ is nonzero, hence for any alternating product $\delta$ of $\beta$ and $\gamma := \alpha_0\alpha_1 \ldots \alpha_{k-1} \gamma_0 \overline{\alpha_{k-1}} \ldots \overline{\alpha_0}$, we have $N^{\alpha\otimes e}_{\delta,\alpha \otimes e}\neq0$, contradicting cofiniteness. Therefore we have proven the claim.

The claim shows that $\alpha \otimes e \in J$ for all $\alpha\in I$. We now show that these are mutually distinct. We only need to show $N^e_{\alpha,e}=0$ for any $\alpha \neq \Unit $. But assume $N^e_{\alpha,e}\neq0$. Since $\alpha \otimes e \in J$, we get $e = \alpha \otimes e$. In particular, $d(\alpha) = 1$. Hence $\alpha \in \Gamma = \underset{s}{*}\Gamma_s$ with  $\Gamma_s\subseteq I_s$ the group of elements in $I_s$ with dimension $1$. Since the $\Gamma_s$ are torsion-free by Proposition \ref{PropSubFus}, also $*_s\Gamma_s$ is torsion-free. As the $\Z_{\Gamma}$-submodule generated by $e$ is cofinite, we deduce that $\alpha =\Unit $.

It now  follows from the above that
 \[ A \to B : \alpha \mapsto \alpha \otimes e\] is a based module isomorphism, finishing the proof.\end{proof}

By contrast, torsion-freeness is in general not preserved by tensor products. Indeed, if $(\Z_I,\otimes)$ is a non-trivial torsion-free fusion ring with $I$ finite, then $\Z_I$ can be made into a non-standard (connected, cofinite) fusion module for $(\Z_{I}\odot\Z_I,\otimes)$ by \[(\alpha\odot \beta)\otimes \gamma = \alpha\otimes \gamma \otimes \overline{\beta}.\] However, this is in a sense the only thing which can go wrong.

\begin{Theorem}\label{TheoTenFus} Let $(\Z_{I_1},\otimes)$ and $(\Z_{I_2},\otimes)$ be torsion-free fusion rings, and assume $(\Z_{I_1}\odot \Z_{I_2},\otimes)$ is not torsion-free. Then $(\Z_{I_1},\otimes)$ and $(\Z_{I_2},\otimes)$ have non-trivial isomorphic finite fusion subrings. 
\end{Theorem}

\begin{proof} Let $(\Z_J,\otimes)$ be a non-zero, non-standard connected cofinite fusion module for $(\Z_I,\otimes) = (\Z_{I_1}\odot \Z_{I_2},\otimes)$. As in the proof of Theorem \ref{TheoFreeProd}, we can find $e\in J$ such that $\alpha\otimes e$ is irreducible for each $\alpha\in I_1\sqcup I_2$, and such that all $\alpha\otimes e$ are mutually distinct for $\alpha \in I_1$ (resp.~ $\alpha\in I_2$).

We claim that $\langle e,e\rangle \neq \Unit $. Indeed, by connectedness we have $\langle c,b\rangle \neq 0$ for each $c,b\in J$. Hence, if $\langle e,e\rangle$ were $\Unit $, then since $\alpha = \langle \alpha\otimes e,e\rangle$ for each $\alpha \in I$, it would follow that $\alpha\otimes e \in J$ for each $\alpha\in I$. Similarly, since $\langle \alpha_1\otimes e,\alpha_2\otimes e\rangle = \alpha_1\otimes \overline{\alpha}_2$, it would follow that the $\alpha\otimes e$ are distinct for distinct $\alpha\in I$. Hence $(\Z_J,\otimes)$ would be the standard $(\Z_I,\otimes)$-module, in contradiction with the assumption. 

There hence exist an $\alpha_1 \in I_1$ and an $\alpha_2\in I_2$, not both the unit element, such that $\langle e,e\rangle = (\alpha_1\otimes \overline{\alpha}_2) \oplus \ldots$, where we identify $(\Z_{I_1},\otimes)$ and $(\Z_{I_2},\otimes)$ as (commuting) fusion subrings of $(\Z_I,\otimes)$. Hence \[0\neq N_{\overline{\alpha}_1\otimes \alpha_2,e}^e = N_{\overline{\alpha}_1,\alpha_2\otimes e}^e  = N_{\alpha_1,e}^{\alpha_2\otimes e}  = N_{\Unit ,\alpha_1\otimes e}^{\alpha_2\otimes e},\] and we deduce $\alpha_1\otimes e = \alpha_2\otimes e$. Since  $\overline{\langle e,e\rangle} = \langle e,e\rangle$, we must have as well $\overline{\alpha}_1\otimes e = \overline{\alpha}_2\otimes e$. 

Let $(\Z_{I_i'},\otimes)$ be the fusion subring of $(\Z_{I_i},\otimes)$ generated by $\alpha_i$. For $w=w_1\ldots w_n$ a word in $\{+,-\}$, let $w(\alpha_i) = \alpha_i^{w_1}\otimes \ldots \otimes \alpha_i^{w_n}$, where $\alpha_i^+ = \alpha_i$ and $\alpha_i^- = \overline{\alpha}_i$. From the above, and from the fact that elements in $I_1$ and $I_2$ commute, we have that $w(\alpha_2)\otimes e = w^{\mathrm{o}}(\alpha_1)\otimes e$, where $w^{\mathrm{o}} = w_n\ldots w_1$. It follows that there exists an identification \[I_1' \cong I_2',\quad \beta_1\leftrightarrow \beta_2\] such that $\beta_1\otimes e = \beta_2\otimes e$. 

Now for each $\beta_1\in I_1'$, we have  \[N_{\overline{\beta}_1\otimes \beta_2,e}^{e} =N_{\overline{\beta}_1,\beta_2\otimes e}^e = N_{\overline{\beta}_1,\beta_1\otimes e}^e = N_{\beta_1,e}^{\beta_1\otimes e} = N_{\Unit ,\beta_1\otimes e}^{\beta_1\otimes e}=1.\]  By cofiniteness, we deduce that $I_1'$ and $I_2'$ are finite. 

It follows that the $(\Z_{I_i'},\otimes)$ are finite fusion subrings $(\Z_{I_i},\otimes)$. If $d(\alpha_1)=1$, then we must have $\alpha_1 = \Unit _1$ by Proposition \ref{PropSubFus}. Since $d(\alpha_1) = d(\alpha_2)$, and since either $\alpha_1$ or $\alpha_2$ is not a unit, it thus follows that neither of them is a unit. Hence both $(\Z_{I_i'},\otimes)$ are non-trivial finite fusion subrings. In fact, it is easily seen that $\beta_1 \mapsto \overline{\beta_2}$ is an isomorphism $(\Z_{I_1},\otimes) \rightarrow (\Z_{I_2},\otimes)$.
\end{proof}

Torsion-freeness does also not automatically pass to fusion subrings. Consider for example the fusion subring of $A(1)$ generated by $\mathbf{2}$ and the cofinite based module for it generated by $\mathbf{1}$ inside $A(1)$. We can say something more however under extra assumptions. The following definition is a straightforward generalisation of \cite[Definition 4.1]{VV11}, in its equivalent characterisation given by \cite[Lemma 4.2.d)]{VV11}. 

\begin{Def} Let $(\Z_I,\otimes)$ be a fusion ring with fusion subring $(\Z_{I'},\otimes)$. We call $(\Z_{I'},\otimes)$ a \emph{divisible} fusion subring of $(\Z_I,\otimes)$ if $(\Z_I,\otimes) \cong \oplus (\Z_{I'},\otimes)$ as based right $(\Z_{I'},\otimes)$-modules.
\end{Def} 
Using the involution, we see that we may replace right by left in the above definition. 
 
\begin{Prop}\label{PropDiv} A divisible fusion subring $(\Z_{I'},\otimes)$ of a torsion-free fusion ring $(\Z_I,\otimes)$ is again torsion-free. 
\end{Prop}  
\begin{proof} Assume $(\Z_{J'},\otimes)$ is a non-zero cofinite connected fusion $(\Z_{I'},\otimes)$-module. Choose an identification $(\Z_I,\otimes) \cong \oplus (\Z_{I'},\otimes)$ of based right $\Z_{I'}$-modules, and let $I_0 = \{\alpha_0\}$ be the elements in $I$ corresponding to the units of the components $\Z_{I'}$. Then, by assumption, $\Z_J = Z_I\underset{\Z_{I'}}{\otimes} \Z_{J'}$ is a based $\Z_I$-module with basis $J = \{\alpha_0\otimes b\mid \alpha_0\in I_0,b\in J'\}$, and it is cofinite by the inner product \[\langle \alpha\otimes b,\gamma\otimes d\rangle = \alpha\otimes\langle b,d\rangle \otimes \overline{\gamma}.\] It is clearly also connected. Hence $(\Z_J,\otimes) \cong (\Z_I,\otimes)$ as a based $(\Z_I,\otimes)$-module.

Let $\alpha_0\otimes b_0$ be the element corresponding to $\Unit \in \Z_I$. Then $\alpha\otimes \alpha_0\otimes b_0 \in J$ for each $\alpha\in I$. It follows in particular that $\alpha\otimes \alpha_0\in I$ for each $\alpha\in I$, and so $\overline{\alpha_0} \otimes \alpha_0 = \Unit $. We may thus assume $\alpha_0 = \Unit $, and we find $\Z_{J'} \cong \Unit \underset{\Z_{I'}}{\otimes} \Z_{J'} \cong \Z_{I'}$ as based $(\Z_{I'},\otimes)$-modules.
\end{proof} 

\begin{Rem} It follows from \cite[Proposition 4.3]{VV11} that $A(2)$ is a divisible fusion subring of $A(1)*\Z_{\Z}$, so we obtain an alternative proof of Theorem \ref{TheoA2} by combining Theorem \ref{TheoFreeProd}, Proposition \ref{TheoA1}, Proposition \ref{PropGroup} and Proposition \ref{PropDiv}.

\end{Rem} 

\section{Strong torsion-freeness for discrete quantum groups}

Let $\Gam$ be a discrete quantum group, by which we will, for the sake of convenience, understand a Hopf $^*$-algebra $(\C[\Gam],\Delta)$ with invariant (positive) state $\varphi:\C[\Gam] \rightarrow\C$. In particular, $(\C[\Gam],\Delta)$ is a cosemisimple Hopf algebra. Endow $\C[\Gam]$ with the pre-Hilbert structure $\langle x,y\rangle = \varphi(x^*y)$. 

By \emph{corepresentation} $X$ of $(\C[\Gam],\Delta)$ we will always understand \emph{finite-dimensio-nal} corepresentation, that is, a finite-dimensional vector space $V$ together with a linear map $\delta: V \rightarrow V\otimes \C[\Gam]$ such that \[(\id\otimes \Delta)\delta = (\delta\otimes \id)\delta, \quad (\id\otimes \varepsilon)\delta = \id.\] We call a corepresentation $V=(V,\delta)$ \emph{unitary} if $V$ is equipped with a Hilbert space structure for which $\delta$ is \emph{isometric}, that is \[\delta(\xi)^*\delta(\eta) = \langle \xi,\eta\rangle 1 \in \C\lbrack \Gam\rbrack\] for all $\xi,\eta\in V$, where we write $\xi^*\eta = \langle \xi,\eta\rangle$. 

The unitary corepresentations form a rigid tensor C$^*$-category \cite{LR97,NT13} with unit the trivial corepresentation $\Unit = \C$. We will write the morphism spaces of $\Gam$-equivariant linear maps as $\Mor^{\Gam}(V,W)$.

\begin{Def}\label{DefFusQG} Let $\Gam$ be a discrete quantum group. We associate to $\Gam$ the fusion ring $\Fus(\Gam)$ with basis the set $I= \{[V]\}$ of equivalence classes of irreducible (unitary) corepresentations, with unit $\Unit  =[\Unit ]$ and involution $\overline{[V]} =[\overline{V}]$, where $\overline{V}$ is the dual of $V$, and fusion rules and dimension function \[N_{[V],[W]}^{[Z]} = \dim(\Mor^{\Gam}(Z,V\otimes W)),\qquad d([V]) = \mathrm{dim}(V).\] 
\end{Def}

See for example \cite{Ban99}, \cite[Section 2.7]{NT13} for more information. 

If $(A,\alpha)$ is a corepresentation which also has the structure of a C$^*$-algebra for which $\alpha$ is a unital $^*$-homomorphism, we call $(A,\alpha)$ a \emph{coaction} of $(\C[\Gam],\Delta)$. The following lemma is closely related to the discussion on $\delta$-forms in \cite{Ban02}.

\begin{Lem}\label{LemQ} If $(A,\alpha)$ is a coaction, there exists a positive functional $\varphi_A$ on $A$ such that $\langle a,b\rangle = \varphi_A(a^*b)$ turns $(A,\alpha)$ into a unitary corepresentation for which the multiplication map \[m: A\otimes A \rightarrow A,\quad a\otimes b\mapsto ab\] is a coisometry.
\end{Lem}
\begin{proof} Let $\omega$ be a positive faithful state on $A$, and let \[\varphi_A'(a) = (\omega \otimes \varphi)\alpha(a).\] From the invariance of $\varphi$, one deduces that $\varphi_A'$ is invariant, that is, \[(\varphi_A'\otimes \id)(\alpha(a)) = \varphi_A'(a)1 \in \C\lbrack \Gam\rbrack.\] Moreover, as we chose $\omega$ faithful, and as $\varphi$ is faithful, also $\varphi_A'$ is faithful. It is then easy to see that $\langle a,b\rangle = \varphi_A'(a^*b)$ turns $(A,\alpha)$ into a unitary corepresentation.  

There now exists a positive, invertible, central element $T \in A$ such that, with respect to the above scalar product, $mm^*(a) = Ta$, namely the \emph{index} of $\varphi_A': A \rightarrow \C$ \cite[Definition 1.2.2., Proposition 1.2.8]{Wat90}. As by definition $m \in \Mor^{\Gam}(A\otimes A,A)$, it follows that $\alpha(T) = T\otimes 1$. Hence $\varphi_A(a) = \varphi_A'(T^{-1}a)$ satisfies the requirements of the lemma. 
\end{proof} 

We will in the following always consider a coaction with an invariant faithful positive functional as above. 

For example, if $(V,\delta)$ is a unitary corepresentation, then $B(V)$ carries the \emph{adjoint coaction} \[B(V) \rightarrow B(V)\otimes \C[\Gam], \quad \xi\eta^* \mapsto \delta(\xi)\delta(\eta)^*.\] 

\begin{Def}[\cite{BS89}] Let $(A,\alpha)$ be a coaction of $(\C[\Gam],\Delta)$. An \emph{equivariant (right) Hilbert $A$-module} is a corepresentation $(\mathscr{E},\delta)$ where $\mathscr{E}$ is equipped with a right Hilbert $A$-module structure for which \[\delta(\xi a) = \delta(\xi)\alpha(a),\quad \langle \delta(\xi),\delta(\eta)\rangle_{A\otimes \C\lbrack \Gam\rbrack} = \alpha(\langle \xi,\eta\rangle_{A}).\] 
\end{Def}

One can turn $(\mathscr{E},\delta)$ into a unitary corepresentation by the inner product $\langle \xi,\eta\rangle = \varphi_A(\langle \xi,\eta\rangle_A)$. Since any Hilbert $A$-module is a direct summand of some $\C^n \otimes A$, it follows that the module maps $\mathscr{E}\otimes A \rightarrow \mathscr{E}$ are coisometries. Moreover, since the space of equivariant $A$-module maps is closed under the adjoint operation, any equivariant Hilbert module decomposes as a direct sum of irreducible equivariant Hilbert modules. 

If $V$ is a unitary corepresentation of $\Gam$, then $V\otimes \mathscr{E}$ is again an equivariant Hilbert $A$-module by the tensor product corepresentation and the $A$-module action only on $\mathscr{E}$. We hence obtain from $(A,\alpha)$ a fusion module for $\Fus(\Gam)$ in the following way. 

\begin{Def} Let $J$ be the set of equivalence classes of irreducible equivariant Hilbert $A$-modules. We define $\Fus_A(\Gam)$ to be the based $\Fus(\Gam)$-module with structure coefficients \[ N_{[V],[\mathscr{E}]}^{[\mathscr{F}]} = \mathrm{dim}(\Mor_A^{\Gam}(\mathscr{F},V\otimes \mathscr{E})).\] 
\end{Def} 

\begin{Lem} The fusion module $\Fus_A(\Gam)$ is cofinite.
\end{Lem} 
\begin{proof} If $\mathscr{E}$ and $\mathscr{F}$ are equivariant Hilbert $A$-modules, then the space of $A$-intertwiners $\mathcal{L}_A(\mathscr{E},\mathscr{F}) \subseteq B(\mathscr{E},\mathscr{F})$ is a unitary corepresentation. This gives rise to the inner product \[\langle [\mathscr{E}],[\mathscr{F}]\rangle = \sum_{[V]\in I}  \mathrm{dim}(\Mor^{\Gam}(V,\mathcal{L}_A(\mathscr{E},\mathscr{F}))[V],\] which is easily seen to be compatible with the $\Fus(\Gam)$-module structure since \[\Mor^{\Gam}(V,\mathcal{L}_A(\mathscr{E},\mathscr{F})) \cong \Mor^{\Gam}_A(V\otimes \mathscr{E},\mathscr{F}).\] 
\end{proof}

\begin{Def}[\cite{Mey08}] A discrete quantum group $\Gam$ is called \emph{torsion-free} if any (finite-dimensional)  coaction is equivariantly Morita equivalent to a direct sum of trivial coactions. 
\end{Def}

More directly, this says a discrete quantum group is torsion-free if and only if any coaction on a finite-dimensional C$^*$-algebra is isomorphic to a direct sum of adjoint coactions. 

\begin{Def}
A discrete quantum group $\Gam$ is called \emph{strongly} torsion-free if $\Fus(\Gam)$ is torsion-free.
\end{Def}

\begin{Theorem}\label{TheoTorFree} A strongly torsion-free discrete quantum group is torsion-free.
\end{Theorem} 

\begin{proof} Let $\Gam$ be a discrete quantum group. Recall that a coaction $(A,\alpha)$ is called \emph{ergodic} if $\Mor^{\Gam}(\Unit ,A) = \C$. By decomposing with respect to the fixed point C$^*$-algebra $A^{\alpha}= \{x\in A\mid \delta(x) = x\otimes 1\}$, one finds that a general coaction $\alpha$ is equivariantly Morita equivalent with $(\oplus_{i=1}^n A_i,\oplus_{i=1}^n \alpha_i)$, with the $(A_i,\alpha_i)$ ergodic coactions. Hence to verify torsion-freeness it is sufficient to check that any \emph{ergodic} coaction is isomorphic to an adjoint coaction.

So, let $A$ be a finite-dimensional C$^*$-algebra with an ergodic coaction $\alpha$ by $(\C[\Gam],\Delta)$. We claim that $\Fus_A(\Gam)$ is connected. Indeed, in this case $A$ itself is an irreducible equivariant Hilbert $A$-module, and any equivariant Hilbert $A$-module $\mathscr{E}$ is contained in $\mathscr{E}\otimes A$, considered as the tensor product of the unitary corepresentation $\mathscr{E}$ and the equivariant Hilbert $A$-module $A$, by the adjoint of the module map $\mathscr{E}\otimes A \rightarrow A$. 

Thus $\Fus_A(\Gam)$ is a non-zero connected cofinite fusion module for $\Fus(\Gam)$. Hence, if we assume $\Gam$ is strongly torsion-free, it is isomorphic to $\Fus(\Gam)$ as a based $\Fus(\Gam)$-module. 

It follows that there exists an irreducible $A$-equivariant Hilbert module $\mathscr{E}$ with $\langle[\mathscr{E}],[\mathscr{E}]\rangle = \Unit $, that is, $\mathcal{L}_A(\mathscr{E}) = \C$. But, by ergodicity of $\alpha$ and invariance and faithfulness of $\varphi_A$, we have for $\xi\neq 0$ that \[(\id\otimes \varphi)(\langle \alpha(\xi),\alpha(\xi)\rangle_{A\otimes \C\rbrack \Gam\rbrack}) = \varphi_A(\langle \xi,\xi\rangle_A)1\neq0,\] hence $\langle \mathscr{E},\mathscr{E}\rangle_A = A$. It follows that $\mathscr{E}$ establishes an equivariant Morita equivalence between $A$ and the trivial coaction on $\C$, hence $(A,\alpha)$ is isomorphic to an adjoint coaction.
\end{proof}

\begin{Cor}\label{CorStrongTor} The following discrete quantum groups are strongly torsion-free, and hence torsion-free.
\begin{enumerate}
\item The dual of a free unitary quantum group.
\item The free product of strongly torsion-free discrete quantum groups.
\end{enumerate}
\end{Cor} 
\begin{proof} It is well-known that the fusion ring of a free product is the free product of the fusion rings associated to the factors \cite{Wan95}. Hence, by Theorem \ref{TheoFreeProd}, a free product of strongly torsion-free discrete quantum groups is again torsion-free. 

For the first point, notice that, by \cite{Wan02}, a general free unitary quantum group (in the sense of \cite{VDW96}) is a free product of free unitary quantum groups of the form $A_u(F)$ \cite{Ban97}, which are either function algebras on the circle group and hence have fusion ring $\Z_{\Z}$, or else have fusion ring $A(2)$ \cite{Ban97}. The duals of free unitary quantum groups are thus free by combining the second part of the corollary with Proposition \ref{PropGroup} and Theorem \ref{TheoA2}. 
\end{proof} 

\begin{Rem} For the free orthogonal quantum groups of \cite{VDW96}, one has to be more careful. Indeed, the free orthogonal quantum group associated to a one-dimensional matrix is the torsion group $\Z_2$, and can by \cite{Wan02} appear as a (free) component in a general free orthogonal quantum group. Nevertheless, if the fundamental representation of the free orthogonal quantum group is irreducible of dimension bigger than 2, its associated fusion ring is $A(1)$ \cite{Ban96}. So, in this case, the dual discrete quantum group is strongly torsion-free by Proposition \ref{TheoA1}. We hence find back in a combinatorial way a result which was proven algebraically in \cite{Voi11}.  
\end{Rem} 

Corollary \ref{CorStrongTor} leaves open the question as to whether the free product of torsion-free discrete quantum groups is again torsion-free. This question has a positive answer, but its proof is more natural within the context of module C$^*$-categories, which will be treated in the next section, see Theorem \ref{TheoFreeProd2}.

To end this section, we prove that strong torsion-freeness is also preserved by Cartesian products. For this, we will need the following lemma.

\begin{Lem}\label{LemFinSub} Let $\Gam$ be a discrete quantum group, and assume $\Gam$ has a non-trivial finite discrete quantum subgroup. Then $\Gam$ is not torsion-free.
\end{Lem} 
\begin{proof} Let $\C\lbrack \Lam \rbrack \subseteq\C\lbrack \mathbbm{\Gam}\rbrack$ be a finite-dimensional Hopf C$^*$-subalgebra. Then \[\Delta: \C\lbrack \Lam \rbrack  \rightarrow \C\lbrack \Lam\rbrack \otimes \C\lbrack \Lam\rbrack  \subseteq \C\lbrack \Lam\rbrack \otimes \C\lbrack \mathbbm{\Gam}\rbrack \] determines an ergodic coaction of $\Gam$. Assume that it is isomorphic to an adjoint coaction associated to a corepresentation $(V,\delta)$. Then the counit on $\C\lbrack \Lam\rbrack$ gives a character on $B(V)$, hence $V$ is one-dimensional and $\C\lbrack \Lam\rbrack = \C$. 
\end{proof}

\begin{Cor} Let $\Gam_1$ and $\Gam_2$ be strongly torsion-free discrete quantum groups. Then $\Gam_1\times \Gam_2$ is also strongly torsion-free.
\end{Cor} 

\begin{proof} If $\Gam$ is a discrete quantum group, a finite fusion subring of $\Fus(\Gam)$ necessarily arises from a finite quantum subgroup of $\Gam$.  As the fusion ring of a Cartesian product of discrete quantum groups is the tensor product of the fusion rings associated to the components \cite{Wan95}, the corollary follows from Theorem \ref{TheoTenFus} and Lemma \ref{LemFinSub}.
\end{proof} 

\section{Torsion-freeness for tensor C$^*$-categories}

In this section, we introduce torsion-freeness in the general setting of rigid tensor C$^*$-categories \cite{LR97,NT13}.  We always assume that our tensor categories are strict with irreducible unit.

\begin{Def} Let $(\mathscr{C},\otimes)$ be a rigid tensor C$^*$-category. A \emph{$Q$-system} \cite{LR97} in $\mathscr{C}$ (or \emph{C$^*$-algebra internal to $\mathscr{C}$}) consists of an associative algebra $(A,m,\eta)$ in $\mathscr{C}$ with $m: A\otimes A \rightarrow A$ the multiplication and $\eta: \Unit \rightarrow A$ the (non-trivial) unit, and such that moreover $m$ is co-isometric.

We call $A$ \emph{ergodic} if $\mathrm{dim}(\mathrm{Mor}(\Unit ,A)) = 1$.
\end{Def} 

\begin{Exa} Let $X$ be an object of $(\mathscr{C},\otimes)$. Let $R: \Unit \rightarrow \overline{X}\otimes X$ and $\overline{R}: \Unit \rightarrow X\otimes \overline{X}$ be solutions to the conjugate equations. Then $X\otimes \overline{X}$ is a $Q$-system by \[m=\frac{1}{\|R\|}(\id\otimes R^*\otimes \id),\quad \eta = \|R\|\overline{R}.\] It is ergodic if and only if $X$ is irreducible. 
\end{Exa} 

If $\Gam$ is a discrete quantum group, any coaction $(A,\alpha)$ on a finite-dimensional C$^*$-algebra can be made into a $Q$-system by Lemma \ref{LemQ}. Conversely, any $Q$-system arises in this way. We first introduce some terminology which will be needed also later on.

\begin{Def} Let $A=(A,m,\eta)$ be a $Q$-system in a rigid tensor C$^*$-category $(\mathscr{C},\otimes)$. A \emph{unitary (right) module} for $A$ is a right $A$-module $(X,n)$ in $(\mathscr{C},\otimes)$ with $n:X\otimes A\rightarrow X$ a co-isometry.
\end{Def} 

One similarly introduces unitary left modules and unitary bimodules. 

The following proposition is probably well-known, but we could not find a direct proof in the literature. 

\begin{Prop}\label{PropQGam} Let $\Gam$ be a discrete quantum group. Then each $Q$-system in $\Corep(\Gam)$ arises from a coaction of $(\C\lbrack \Gam\rbrack,\Delta)$. 
\end{Prop} 
\begin{proof} If $(A,m,\eta)$ is a $Q$-system in $\Corep(\Gam)$, this immediately gives us a unital algebra $A$ with a coaction $\alpha: A\rightarrow A\otimes \C\lbrack \Gam\rbrack$. It remains to show that $A$ is a C$^*$-algebra and that $\alpha$ is $^*$-preserving.

Since $A$ is also a finite dimensional Hilbert space, we can make sense of the adoint $L_a^*$ of the operator of left multiplication with $a\in A$. Define then \[a^* = L_a^*1.\] Since $m$ is an $A$-bimodule map, and since the adjoint of an $A$-bimodule map is still an $A$-bimodule map by identity (d) in \cite[Section 6]{LR97}, we find \begin{multline*}\langle a,b^*c \rangle = \langle m^*a,L_b^*1\otimes c\rangle = \langle (b\otimes 1)m^*a,1\otimes c \rangle \\= \langle m^*(ba),c\rangle = \langle ba,c\rangle = \langle a, L_b^*c\rangle.\end{multline*}
Hence $L_a^* = L_{a^*}$, and $A$ is a C$^*$-algebra. 

Write $\varphi_A(a) = \langle 1,a\rangle$. Then we see $\varphi_A(a^*b) = \langle a,b\rangle$, and hence $\varphi_A$ is a faithful positive functional on $A$. Moreover, since $\alpha$ is a unitary corepresentation, we find \[(\varphi_A\otimes \id)(\alpha(a)^*\alpha(b)) = \varphi_A(a^*b)1.\] Hence \[U: A\otimes \C\lbrack \Gam\rbrack \rightarrow A\otimes \C\lbrack \Gam\rbrack,\quad a\otimes g\mapsto \alpha(a)(1\otimes g)\] is unitary. An easy computation shows $U(L_a\otimes 1)U^* = L_{\alpha(a)}$, and hence $\alpha(a^*) = \alpha(a)^*$. 
\end{proof}

From now on, by module we will mean \emph{right} module. 

\begin{Lem} Let $A=(A,m,\eta)$ be an ergodic $Q$-system. Then the category of unitary $A$-modules is a C$^*$-category.
\end{Lem} 

The lemma remains true for arbitrary $Q$-systems, but we will not need this result.

\begin{proof} The only thing which is not clear is if the adjoints of $A$-linear morphisms are again $A$-linear. The argument for this is similar to \cite[Lemma 6.1]{LR97}. 

Namely, fix a unitary module $(X,n)$ for $A=(A,m,\eta)$. Recall that, by means of the standard solutions, one can form \emph{partial traces} \cite[Section 3]{LR97}, \cite[Section 2.5]{NT13}\[(\id\otimes \mathrm{Tr}_Y): \End(X\otimes Y) \rightarrow \End(X)\] which are faithful, completely positive $\End(X)$-linear maps. Now as $m^*m$ is an $A$-bimodule map from $A\otimes A$ to $A\otimes A$, it follows that $T= (\id\otimes \mathrm{Tr}_A)m^*m$ is a (left) $A$-module map $A\rightarrow A$. Hence $T = m(\id_A\otimes (T\circ \eta))$. By ergodicity however, $T\circ \eta = c\eta$ for some $c>0$, and so $T = c\id_A$. On the other hand, $\mathrm{Tr}_{A\otimes A}(m^*m) = \mathrm{Tr}_A(mm^*) = \mathrm{Tr}_A(\id_A)$, so $c=1$. 

Now a small computation reveals that, with $r = (1_X\otimes m)(n^*\otimes 1_A)-n^*n$, \[r^*r = (n\otimes 1_A)(1_X\otimes m^*m)(n^*\otimes 1_A)-n^*n\geq0.\] As \begin{multline*}\mathrm{Tr}_{X\otimes A}\left((n\otimes 1_A)(1_X\otimes m^*m)(n^*\otimes 1_A)-n^*n\right) \\= \mathrm{Tr}_{X}(nn^*)- \mathrm{Tr}_{X\otimes A}(n^*n)=0,\end{multline*} we deduce that $r=0$, and so \[(1_X\otimes m)(n^*\otimes 1_A)=n^*n.\]

We deduce that an $f \in \End(V)$ lies in $\End_A(V)$ if and only if $n(f\otimes 1)n^* =f$, and hence $\End_A(V)$ is a $^*$-algebra. 
\end{proof}

\begin{Cor} The category of unitary right modules is a (strict, left) $(\mathscr{C},\otimes)$-module C$^*$-category \cite{DCY13} by \[V\otimes (X,n) = (V\otimes X,\id_V\otimes n).\]
\end{Cor} 

A $Q$-system of the form $X\otimes \overline{X}$ will be called \emph{trivial}. 

\begin{Def} A rigid tensor C$^*$-category will be called \emph{torsion-free} if each ergodic $Q$-system is trivial.
\end{Def}

This definition is compatible with the definition for discrete quantum groups by Proposition \ref{PropQGam}. It is also not hard to show that in a torsion-free rigid C$^*$-category, every $Q$-system is isomorphic to a direct sum of trivial $Q$-systems.

Recall that to any rigid tensor C$^*$-category $(\mathscr{C},\otimes)$ can be associated the fusion ring $\Fus(\mathscr{C})$, determined as in Definition \ref{DefFusQG}, with the dimension function now given by the intrinsic dimension \cite{LR97}. 

\begin{Def} We call a rigid tensor C$^*$-category $(\mathscr{C},\otimes)$ \emph{strongly torsion-free} if $\Fus(\mathscr{C})$ is torsion-free. 
\end{Def} 

We will show that Theorem \ref{TheoTorFree} still holds in this setting. Note that associated to a (strict) module C$^*$-category $(\mathscr{D},\otimes)$ one has the based $\Fus(\mathscr{C})$-module $\Fus(\mathscr{D})=(\Z_J,\otimes)$ where $J$ is the set of equivalence classes of irreducible objects in $\mathscr{D}$.

\begin{Def} A module C$^*$-category $(\mathscr{D},\otimes)$ over $(\mathscr{C},\otimes)$ is called \emph{cofinite} (resp.~ \emph{connected}) if $\Fus(\mathscr{D})$ is a cofinite (resp.~ connected) based $\Fus(\mathscr{C})$-module.
\end{Def}

\begin{Lem}\label{LemTriv} Assume $(\mathscr{D},\otimes)$ is a module C$^*$-category over $(\mathscr{C},\otimes)$, and assume $\Fus(\mathscr{D}) \cong \Fus(\mathscr{C})$ as based modules. Then $(\mathscr{D},\otimes) \cong (\mathscr{C},\otimes)$ as module C$^*$-categories. 
\end{Lem} 

\begin{proof} By assumption, we can find an element $\Unit  \in \mathscr{D}$ such that $X\otimes \Unit $ is irreducible for each $X$. As the functor $X\mapsto X\otimes \Unit $ also essentially surjective, it follows that $X\rightarrow X\otimes \Unit $ is an equivalence of C$^*$-categories. As we assume strictness, it is immediate that this is a (strict) equivalence of module C$^*$-categories. 
\end{proof}

\begin{Lem}\label{LemModEqQ} A rigid tensor C$^*$-category is torsion-free if and only if each non-trivial connected cofinite module C$^*$-category is equivalent to the tensor C$^*$-category as a module C$^*$-category.
\end{Lem} 

\begin{proof} Let $(\mathscr{C},\otimes)$ be a rigid tensor C$^*$-category, and let $(\mathscr{D},\otimes)$ be a non-trivial connected cofinite module C$^*$-category. Recall the internal $\Mor$-functor of \cite{Ost03b} to make sense of $\underline{\Mor}(M,N) \in \mathscr{C}$ for objects $M,N\in \mathscr{D}$, defined by the identity \[\Mor_{\mathscr{C}}(X,\underline{\Mor}(M,N) \cong \Mor_{\mathscr{D}}(X\otimes M,N).\] Then with $M_0 \in \mathscr{D}$ a fixed irreducible object, $A = \underline{\mathrm{End}}(M_0)$ is in a natural way an ergodic $Q$-system \cite[Lemma 2.18]{GS12}. Moreover, by \cite[Theorem 1]{Ost03b}, $(\mathscr{D},\otimes)$ is equivalent to the module C$^*$-category $(\mathscr{D}',\otimes)$ of right unitary $A$-modules by means of the correspondence \[M \mapsto \underline{\Mor}(M_0,M),\qquad \underline{\Mor}(M_0,M)\otimes \underline{\End}(M_0)\rightarrow \underline{\Mor}(M_0,M).\] Here the assumption of finitiness in \cite[Theorem 1]{Ost03b} is easily seen to be replaceable by cofiniteness. Also, the unitarity of the above module is easily proven be a direct argument involving $2$-C$^*$-category language, as is the fact that the above is then a $^*$-functor and hence a $^*$-equivalence.

If now $(\mathscr{C},\otimes)$ is torsion-free, it follows that $A \cong X\otimes \overline{X}$, and we then have $(\mathscr{C},\otimes)\cong (\mathscr{D}',\otimes)$ as module C$^*$-categories by $Y\mapsto Y\otimes \overline{X}$. 

Conversely, let $A$ be an ergodic $Q$-system in $(\mathscr{C},\otimes)$, and let $(\mathscr{D},\otimes)$ be the associated module C$^*$-category of unitary $A$-modules. It is easy to see that we have an isomorphism of $Q$-systems $A \cong \underline{\End}_A(A)$, corresponding to the map $m\in \Mor_A(A\otimes A,A)$. By assumption, $(\mathscr{D},\otimes)$ is equivalent to $(\mathscr{C},\otimes)$ as a module C$^*$-category. Let $X\in \mathscr{C}$ be the irreducible object corresponding to the irreducible unitary $A$-module $A$ under this equivalence. Then \[A\cong \underline{\End}_A(A)\cong \underline{\End}_{\,\Unit }(X)\cong X\otimes \overline{X}.\] 
\end{proof} 

\begin{Theorem}\label{TheoTorTor} A rigid tensor C$^*$-category is torsion-free if it is strongly torsion-free.
\end{Theorem} 

\begin{proof} By Lemma \ref{LemModEqQ}, it is sufficient to prove that each non-trivial connected cofinite module C$^*$-category is isomorphic to $(\mathscr{C},\otimes)$. However, this follows immediately from Lemma \ref{LemTriv}. 
\end{proof} 

\begin{Rem}\label{RemRootUn}
It is known that $\Z[\phi]$ appears as the fusion ring of a (unique) rigid tensor C$^*$-category, the \emph{Fibonacci tensor C$^*$-category} \cite{MS89,Ost03a}. By Proposition \ref{PropTorFree} and Theorem \ref{TheoTorTor}, it follows that this (finite!) tensor C$^*$-category is strongly torsion-free. In fact, this tensor C$^*$-category is the representation category of the even part of the quantum group at $5$th root of unity $SU(2)_3$, and one can show the even parts $SU(2)_{N,\mathrm{even}}$ of the quantum groups at root of unity $SU(2)_N$ are strongly torsion-free at all \emph{odd} levels $N$, using for example the fusion ring homomorphism $\Fus(SU(2)_N) \rightarrow \Fus(SU(2)_{N,\mathrm{even}})$ which sends the spin $\frac{k}{2}$ representation for odd $k$ to the spin $\frac{N-k}{2}$-representation and which is the identity on the integer spin representations, together with the classification of based $\Fus(SU(2)_N)$-modules in \cite{EK95}. 
\end{Rem}

Unlike the case of fusion rings, there also exist non-trivial finite torsion-free tensor C$^*$-categories with integer dimension function.

\begin{Exa} Consider the \emph{semion category}, that is, the rigid tensor C$^*$-category $(\mathscr{C},\otimes)$ of super-Hilbert spaces $\mathscr{H} = \mathscr{H}_0\oplus \mathscr{H}_1$ with the ordinary tensor product but with the non-trivial associator \[a_{\mathscr{H}_i\otimes \mathscr{H}_j\otimes \mathscr{H}_k} = (-1)^{ijk}.\] Then with $C_2$ the finite simple group of order two, $\Fus(\mathscr{C}) = \Fus(C_2) =\Z_I$ with $I= \{+,-\}$ where $+= \Unit $ and $-\otimes - = +$, and with trivial involution.

If $\mathscr{M}$ is a non-trivial connected cofinite module C$^*$-category over $\mathscr{C}$, and $J = \mathrm{Fus}(\mathscr{M})$, then $C_2$ acts transitively on $J$, and so $|J|=1$ or $|J|=2$. 

If $|J|=1$, then $\mathscr{M}$ is the category $\mathrm{Hilb}$ of Hilbert spaces, which means we have a concrete representation of $(\mathscr{C},\otimes)$ inside $\mathrm{Hilb}$. This is a contradiction \cite[Appendix E]{MS89}, \cite[Example 2.3]{EG04}. 

Hence $|J|=2$ and $C_2$ acts nontrivially on $J$. This means $(\Z_J,\otimes)\cong (\Z_I,\otimes)$ as based modules, and hence $(\mathscr{C},\otimes)\cong (\mathscr{M},\otimes)$ as module C$^*$-categories. 
\end{Exa}

\begin{Rem}
In fact, this tensor C$^*$-category is the representation category of the quantum group at $3$rd root of unity $SU(2)_1$, and by arguments as in the classification of `quantum subgroups of $SU(2)_N$' \cite{Oc02,KO02,DCY15}  one can see that all quantum groups at root of unity $SU(2)_N$ are torsion-free at all \emph{odd} levels $N$. Nevertheless, they are not strongly torsion-free by the existence of the non-trivial fusion ring homomorphism $\Fus(SU(2)_N) \rightarrow \Fus(SU(2)_{N,\mathrm{even}})$ (see Remark \ref{RemRootUn}).
\end{Rem}

We do not have any examples however of torsion-free discrete quantum groups which are not strongly torsion-free.

As an application of the theory developed in this section, we can give an easy proof of the following theorems.

\begin{Theorem}\label{TheoFreeProd2} Let $\{\Gam_s\mid s\in S\}$ be a collection of torsion-free discrete quantum groups. Then $\Gam = *_{s\in S}\Gam_s$ is torsion-free. 
\end{Theorem} 
\begin{proof} Let $(\mathscr{C},\otimes) = \Corep(\Gam)$, and let $(\mathscr{D},\otimes)$ be a non-trivial connected, cofinite module C$^*$-category. Write $(\Z_I,\otimes) = \Fus(\Gam)$, and $(\Z_J,\otimes) = \Fus(\mathscr{D})$. As mentioned already, $\Fus(\Gam) = *_{s\in S} \Fus(\Gam_s)$. 

Let $X$ be an irreducible object in $\mathscr{D}$, and let $(\mathscr{D}_s^{X},\otimes)$ be the (connected, cofinite) module C$^*$-category for $(\mathscr{C}_s,\otimes) = \Corep(\Gam_s)$ generated by $X$. Then, in the notation of the proof of Theorem \ref{TheoFreeProd}, $\Fus(\mathscr{D}_s^X) = (B_s^{\lbrack X\rbrack},\otimes)$. However, by assumption of the torsion-freeness of $\Gam_s$, we have $(\mathscr{D}_s^X,\otimes) \cong (\mathscr{C}_s,\otimes)$ as a module C$^*$-category. Hence $(B_s^{\lbrack X\rbrack},\otimes) \cong (\Z_{I_s},\otimes)$ as a fusion module. 

The proof of Theorem \ref{TheoFreeProd} can now be followed ad verbatim to conclude that $\Fus(\mathscr{D}) \cong \Fus(\mathscr{C})$ as fusion modules. By Lemma \ref{LemTriv} and Lemma \ref{LemModEqQ}, it follows that $(\mathscr{C},\otimes)$ is torsion-free, and hence $\Gam$ is torsion-free.
\end{proof} 

\begin{Theorem}\label{TheoTensProd2} Let $\Gam_1,\Gam_2$ be torsion-free discrete quantum groups. Then $\Gam_1\times \Gam_2$ is torsion-free.
\end{Theorem} 

\begin{proof} Assume $\Gam_1,\Gam_2$  are torsion-free. Then one again sees that the proof of Theorem \ref{TheoTenFus} can be copied ad verbatim to conclude that either $\Gam_1\times \Gam_2$ is torsion-free or $\Gam_1$ has a non-trivial finite quantum subgroup. However, the latter is excluded by Lemma \ref{LemFinSub}. 
\end{proof}

\end{document}